\documentclass[12pt,a4paper]{amsart}
\usepackage{amssymb}
\usepackage{amsmath,amsthm}
\usepackage{a4wide}
\newtheorem{theorem}{Theorem}[section]

\newtheorem{proposition}[theorem]{Proposition}

\theoremstyle{definition}
\newtheorem{definition}[theorem]{Definition}

\newtheorem{corollary}[theorem]{Corollary}
\newcommand{\supp}{\mathop{\mathrm{supp}}}

\def\C{\mathbb C}
\def\N{\mathbb N}
\def\S{\mathcal{S}}

\newcommand{\R}{\mathbb{R}}

\newcommand{\al}{\alpha}
\newcommand{\be}{\beta}
\newcommand{\NN}{\mathbb N}

\newcommand{\RR}{\mathbb R}

\newcommand{\DD}{\mathcal D}

 \newcommand{\bea}{\begin{eqnarray}}
 \newcommand{\eea}{\end{eqnarray}}

 \newcommand{\beas}{\begin{eqnarray*}}
 \newcommand{\eeas}{\end{eqnarray*}}

\theoremstyle{remark}
\newtheorem{remark}[theorem]{Remark}

\newcommand{\beq}{\begin{eqnarray}}
\newcommand{\eeq}{\end{eqnarray}}

\newcommand{\beqs}{\begin{eqnarray*}}
\newcommand{\eeqs}{\end{eqnarray*}}

\numberwithin{equation}{section}

\theoremstyle{remark}

\usepackage{color}

\title[Directional short-time Fourier transform of ultradistributions]{Directional short-time Fourier transform of ultradistributions}
\author[S. Atanasova]{Sanja Atanasova}
\address{Faculty of Electrical Engineering and Information Technologies, Ss. Cyril and Methodius University\\  Rugjer Boshkovik 18\\
1000 Skopje\\ North Macedonia}
\email{ksanja@feit.ukim.edu.mk}
\author[S. Maksimovi\'{c}]{Snje\v{z}ana Maksimovi\'{c}}
\address{Faculty of Architecture, Civil Engineering and Geodesy, University of Banja Luka\\ Stepe Stepanovi\'{c}a 77/3 \\ 78000 Banja Luka\\ Bosnia and Hercegovina}
\email{snjezana.maksimovic@aggf.unibl.org}
\author[S. Pilipovi\'{c}]{Stevan Pilipovi\'{c}}
\address{Faculty of Sciences and Mathematics, University of Novi Sad\\ Trg D. Obradovica 4\\ 21000 Novi Sad\\ Serbia}
\email{pilipovic@dmi.uns.ac.rs}


\begin{document}

\begin{abstract}
We define  and analyse the $k$-directional short-time Fourier transform and its synthesis operator over Gelfand Shilov spaces $\S^\al_{\be}(\R^n)$ and  $\S^\al_{\be}(\mathbb R^{k+n})$ respectively, and their duals. Also, we investigate
 directional regular sets and their complements - directional wave fronts, for elements of
 $\S^{\prime \al}_{\al}(\R^n)$.
\end{abstract}

\maketitle

\section{Introduction}

 Grafakos and Sansing \cite{grafakos} developed the concept of directional sensitive variant of the short-time Fourier transform (STFT) by introducing the Gabor ridge functions  that can be viewed  as a time-frequency analysis  in a  certain domain. A  slightly different version of their concept was considered by Giv in \cite{Giv}, where he defined the directional short-time Fourier transform (DSTFT).   In \cite{KS} the authors analyzed the DSTFT on Schwartz test function spaces, proving the continuity theorems and extending it to  the spaces of tempered distributions. Moreover, in \cite{APS}, the results of Giv are extended through the investigations of the STFT in the direction of $u\in\mathbb S^{n-1}$ on the distributions of exponential type and through  the analysis of the (multi)directional  wave front sets for tempered distributions.

In the past several decades, there has been a trend in investigating integral transforms on the spaces of ultradistribution, such as the wavelet transform,  STFT, Laplace and Hilbert transform, \cite{CKP, DP, KPP, Pil1, T1}, as natural generalization of these transforms over the  the spaces of distributions. In the first part of this paper, in Section \ref{se2}, we introduce on  Gelfand-Shilov spaces (of Roumieu and Beurling type) \cite{GS} the multidimensional STFT in the direction of $\bold u^k=(u_1,...,u_k)$, where $u_i, i=1,\ldots,k$ are independent vectors of $\mathbb S^{n-1}$.  Moreover, we analyze the corresponding synthesis operator. By a linear  transformation of coordinates, we simplify our exposition considering direction  $\bold e^k=(e_1,...,e_k)$. By the continuity results presented in Theorem \ref{ridir} and \ref{synthcont}, we show in Subsection \ref{temperedultra} that both transforms defined as a transposed mapings and as actions on appropriate window functions, coincide on  Gelfand Shilov (GS) spaces  of ultradistributions.

In the second part, in Section \ref{se3}, we analyze  the regularity properties of a GS ultradistribu- tions $f\in\mathcal S^{'\al}(\mathbb R^n)=\mathcal S^{\prime\al}_\al(\mathbb R^n)$ by introducing the $k$-directional regular sets and the wave front sets using the $k$-DSTFT $f$.  The main result is presented in Theorem \ref{nwr}, where we show that the $k$-directional wave front does not  depend on the window function. We also consider partial wave front in the sense of \cite{hor} and show that this notion is equivalent with the $k$-directional wave front. Actually, the partial wave front was  not considered  neither  for distribution spaces nor for GS type spaces.
More on the micro-local analysis in spaces of ultradistributions can be found in \cite{JPTT, JPTT1}, as well as in \cite{Coriasco1, Coriasco2}   and the references there in.

Let  us note that we follow  \cite{APS} (related to distributions) and give here new details which are important when developing the analysis on GS spaces. Similarly, in the last subsection when we compare directional wave fronts, we  follow \cite{PP}  but again with details related to ultradistributions which makes the proof more complex.

\subsection{Notation}
We employ the  notation $\N_0$, $\R$ and $\C$ for the sets of natural (including zero), real and complex numbers, respectively; $\mathbb{S}^{n-1}$ stands for the unit sphere of $\R^n$.
For given multi-indexes $p=(p_1,...,p_n),v=(v_1,\ldots,v_n)\in\N_0^n$ and  $t=(t_1,...,t_n)\in\R^n$, we write $$t^p=t_1^{p_1}\cdot...\cdot t_n^{p_n}, \,
(-i)^{|p|}D^p=\partial^p_t=\frac{\partial^{|p|}}{\partial t_{1}^{p_1}\cdots \partial t_{n}^{p_n}}, \, \, |p|=p_1+...+p_n,
$$
 and
for $v\leq p, \, \binom{p}{v}=\prod_{i=1}^n\binom{p_i}{v_i}$,  where $ v\leq p \ \text{means} \ v_i\leq p_i, \ i=1,...,n.$
%
Points in $\mathbb R^k$ are denoted by $\tilde x, \tilde y,...,$ while   points in $\mathbb R^{n-k}$ are denoted by $\tilde{\tilde x}, \tilde{\tilde y},\ldots$. So $$ x=(x_1,...,x_n)=(\tilde x,\tilde{\tilde x}), \mbox{ where } \tilde x=(x_1,...,x_k),\;  \tilde{\tilde x}=(x_{k+1},...,x_n).$$
Elements of $\mathbb N_0^k$ and $\mathbb N_0^{n-k}$ are denoted by tilde and double tilde as in   multidimensional  real cases.
For an open set $\Omega\subseteq\R^n$ the symbol $K\subset\subset\Omega$ means that $K$ is a compact set contained in $\Omega$. The support of
 a given function (or ultradistribution) $f$ is denoted by $\supp f$.
 We say that a function $f$ is compactly supported if there exists
 a $K\subset\subset\R^n$ such that $\supp f\subset K$. The Fourier transform of a function $f$ is defined as $\hat f(\xi)=\int_{\R^n}f(x)e^{-2\pi i\xi\cdot x}dx$, $\xi\in\R^n$.  For the dual pairing we use $\langle f,g\rangle$;  for the $L^2$ inner product of $f$ and $g$ we use the notation $(f,g)$.

\subsection{Ultradistribution spaces}
Following the approach of \cite{Kom73},  we introduce  test spaces   defined by the use of  Gevrey sequences $ M_p=p!^\alpha, \alpha>1$. All the results can be given for general sequences $(M_p)_{p\in\N}$ satisfying appropriate conditions. We only consider   spaces  of Roumieu type \cite{Kom73}. Moreover, the results of  this paper also hold for the Beurling type spaces. The  topology for them is the same  as in the case of Schwartz space of rapidly decreasing functions. Because of  that, we omit this part of analysis.

Let $K\subset\subset\Omega$ and $h>0$. We recall the definitions
of some spaces of test functions \cite{Kom73} ({of the Roumieu type}):
\begin{align*}
    \mathcal{E}^{\al}_h(K)\, &:=\, \{\varphi\in\mathcal{C}^{\infty} \left(\Omega\right)\!:\
    \sup_{{t\in K, p\in\N_0^n}}
    \frac{h^{ |p|}}{p!^\al}|D^p\varphi(t)|<\infty\};
    \\
   \mathcal{D}^{\al}_h(K)\, &:=\,
    \mathcal{E}^{\al}_h(K)\cap\{\varphi\in\mathcal{C}^{\infty}\left(\Omega\right)\!:\
    \supp\varphi\subset K\};\\
    \\
    \mathcal{E}^{\al} (K)\, &:=\,
    \varinjlim_{{h\rightarrow0}}\mathcal{E}^{\al}_h(K); \qquad
    \mathcal{E}^{\al}(\Omega)\, :=\, \varprojlim_{K\subset\subset\Omega} \mathcal{E}^{\al} (K);\\
    \mathcal{D}^{\al} (K)\, &:=\,
    \varinjlim_{{h\rightarrow0}}\mathcal{D}^{\al}_h(K); \qquad
    \mathcal{D}^{\al}(\Omega)\, :=\, \varinjlim_{K\subset\subset\Omega} \mathcal{D}^{\al} (K).
  \end{align*}

 The spaces of {ultradistributions} $ \mathcal{D}^{\prime\al}(\Omega)$ is a strong dual of $\mathcal{D}^{\al}(\Omega)$; its subspace $ \mathcal{E}^{\prime\al}(\Omega)$ consists of
compactly supported ultradistributions. All these spaces are complete, bornological,  Montel and Schwartz (\cite{Kom73}).

\subsubsection{Gelfand-Shilov type spaces}

Let $\al,\be,a>0$. By $(\mathcal{S}_a)^{\al}_{\be}(\R^n)$ is denoted the Banach space of all smooth functions $\varphi$ on $\R^n$ such that the norm
\begin{equation}\label{ngs}
\sigma^{\al,\be}_a(\varphi)=\sup_{t\in\R^n,p,q\in\N_0^n}\frac{a^{|p|+|q|}}{p!^\be q!^\al}|t^p \varphi^{(q)}(t)|
\end{equation}
 is finite.
The space $\mathcal{S}^\al_\be(\R^n)$ is defined as an inductive limit of the space $(\S_a)^{\al}_{\be}(\R^n)$:
$$ \mathcal{S}^\al_\be(\R^n)=\varinjlim_{{a\rightarrow0}}(\S_a)^{\al}_{\be} (\R^n).$$
This is a $DFS$ (dual of Fr\' echet Schwartz) nuclear space. Its strong dual is  called Gelfand Shilov space of  ultradistributions.
The space $\mathcal{S}^\al_\be(\R^n)$ is nontrivial if and only if $\al+\be\geq{1}$. The Fourier transform is a topological isomorphism between $\mathcal{S}^\al_\be(\R^n)$ and $\mathcal{S}^\be_\al(\R^n)$, which extends to a continuous linear transform from $\mathcal{S}^{\prime\al}_{\be}(\R^n)$ onto $\mathcal{S}^{\prime\be}_{\al}(\R^n)$.  If $\al=\be$, the  space $\mathcal{S}^\al_\al(\R^n)$ is denoted by $\S^{\al}(\R^n)$.

\begin{remark}\label{re1}We will often use  equivalent families of norms in which in \eqref{ngs}  $t^p \varphi^{(q)}(t)$ is replaced by $(t^p\varphi)^{(q)}$ or $t^q \hat{\varphi}^{(p)}$ or $(t^q\hat\varphi)^{(p)}$ (See  \cite{CKP}, \cite{Pil} and \cite{Pil1}).

\end{remark}

\subsubsection{Ultradifferential operators}\label{ultrap}
A
formal expression $P(D)=\sum_{p\in\N_0^n}
a_p D^p$ ($a_p\in\mathbb{R}$), corresponds to the
ultrapolynomial  $P(\xi)=\sum_{p\in\N_0^n} a_p
\xi^{p}$ ($\xi\in\R^n$), \cite{Kom73}. It is called an
    \emph{ultradifferential operator
    of the Roumieu type} $\alpha$
     if for every $a>0$, there exist a constant $C=C(a)>0$ such that the coefficients $a_p$ satisfy the estimate
    \begin{equation}\label{RP}
|a_p|\, \leq\,
    \frac{C a^{|p|}}{p!^{\al}}, \quad \forall p\in\N_0^n.
   \end{equation}
It is a $C^\infty$ function on $\R^n$.
We will use the following representation theorem for an $f\in
\mathcal{S}^{\prime\alpha}_\beta(\R^n):$

For any $f\in\mathcal{S}^{\prime \al}_\be(\R^n)$ there exist $P_1(D)$-ultradifferential operator of the Roumieu type $\al$, an ultrapolynomial $P_2(x)$ of Roumieu type $\be$ and an $F\in L^2(\R^n)$ so that
\begin{equation}\label{rep11}
f(x)=P_1(D)(P_2(x)F(x)).
\end{equation}
We note that  $\phi\mapsto P_1(D)\phi$ and $\phi\mapsto P_2(x)\phi$ are continuous mappings of $\mathcal S^\al_\be(\R^d)$ into itself.

   We will consider elliptic operators of this type. For them one has that the function $P(\xi)$
satisfies (\cite[Proposition 4.5]{Kom73})
   \begin{equation}\label{ulpobound}
 \left( \forall a,\;\,
   \exists C>0 \right)\;\, \forall \xi\in \R^n \quad |P(\xi )|\leq Ce^{a|\xi|^{1/\al}}.
  \end{equation}
    We point out  that $P(D)$ defines the continuous mappings on $\mathcal{S}^{\al}_{\be}(\R^n)$. Moreover
$$P(D)\varphi=\lim_{n\rightarrow \infty}\sum_{|p|<n}a_p D^p\varphi \,\, \rm{in} \,\, \mathcal{S}^{\prime\al}_{\be}(\R^n)\,\, \rm{for \, every} \, \, \varphi\in\mathcal S^{\al}_\be(\R^n).$$

\subsection{STFT and the synthesis operator on $L^2$} Let a window function $g\in L^2(\R^n)\setminus\{0\}$. Then STFT is defined by
\begin{equation*}\label{stft}
V_gf(y,\xi)=\int_{\R^n}f(t)\overline{g(t-y)}e^{-2\pi i \xi\cdot t}dt,\quad y,\xi\in\R^n,  f\in  L^2(\R^n).
\end{equation*}

The synthesis operator  $V^{*}_g$ is defined on $L^2(\R^{2n})$ by
\[ V^{*}_gF(t)=\int\int_{\R^{2n}}F(y,\xi)g_{y,\xi}(t) dyd\xi, \ \ t\in\R^n\]
where $g_{y,\xi}(t)={g(t-y)}e^{2\pi i \xi\cdot t}$. Let  $\varphi\in L^2(\R^n)$ be a synthesis window for $g$ ($(g,\varphi)\neq 0$). Then for any $f\in L^2(\R^n)$
\begin{equation}\label{as}
f(t)=\frac{1}{(g,\varphi)}\int \int_{\R^{2n}}V_gf(y,\xi)\varphi_{y,\xi}(t) dyd\xi,
\end{equation}
where $\varphi_{y,\xi}(t)=\varphi(t-y)e^{2\pi i  \xi\cdot t}$. It is well-known that if $g\in \mathcal{S}^\al_\be(\R^n)\setminus\{0\}$ is a fixed window, then $V_g:\mathcal{S}^\al_\be(\R^n)\rightarrow \mathcal{S}^{\al}_{\be}(\R^{2n})$
is a continuous mapping. Moreover, for $f\in\mathcal{S}^\al_\be(\R^n)$, equation \eqref{as} holds pointwise (see \cite{T1}).

\section{$k$-directional STFT and the $k$-directional synthesis operator}\label{se2}

We define the $k$-DSTFT and the $k$-directional synthesis operator (DSO)  for a fixed direction, over
$\mathcal{S}^\alpha_{\beta}(\R^n)$ and its dual.
\begin{definition}\label{analsyntopp} Let $\bold u^k=(u_1,\ldots,u_k),$ where $u_i, i=1,\ldots,k$ are independent vectors of $\mathbb S^{n-1}$. Let $\tilde y=(y_1,\ldots,y_k)\in\mathbb R^k$ and $g\in \mathcal{S}^\alpha_{\beta}(\R^k)\setminus\{0\}$. The $k$-DSTFT of  $f\in L^2(\R^n)$ is defined by
\begin{equation}\label{dd2}
DS_{g,\bold u^k} f(\tilde y,\xi):  = \int_{\mathbb R^n} f(t)\overline{g((u_1\cdot t,...,u_k\cdot t)-(y_1,...,y_k))} e^{-2\pi it\cdot \xi}dt
\end{equation}
 and the $k$-DSO of  $F\in L^2(\R^{2n})$ is defined by
\begin{equation}\label{ds}
DS^\ast_{g,\bold u^k} F(t)=\int_{\R^n}\int_{\R^k}F(\tilde y,\xi){g_{\bold u^k, \tilde y,\xi}}(t)d\tilde yd\xi, \quad t\in\R^n,
\end{equation}
where $g_{\bold u^k,\tilde y,\xi}(t)=g((u_1\cdot t,...,u_k\cdot t)-(y_1,\cdots,y_k))e^{2\pi i\xi\cdot t}, t\in\R^n.$
\end{definition}


Let $\varphi\in\mathcal{S}^\al_\be(\R^k)$ be the  synthesis window for  $g\in \mathcal{S}^\al_\be(\R^k)\setminus\{0\}$, which means $(g,\varphi)_{L^2(\R^k)}\neq 0.
 $
We will show in Proposition \ref{propg} that for $f\in \mathcal{S}^\al_\be(\R^n)$ the following reconstruction formula holds pointwise:
\begin{equation}\label{1rfdstft}
f(t)=\frac{1}{(g,\varphi)}\int_{\R^n}\int_{\R^k}DS_{g,\bold u^k}f(\tilde y,\xi)\varphi_{\bold u^k,y,\xi}(t) d\tilde yd\xi,
\end{equation}
where $\varphi_{\bold u^k,y,\xi}(t)=\varphi((u_1\cdot t,...,u_k\cdot t)-(y_1,\ldots, y_n))e^{2\pi i\xi\cdot t}, \ t\in\R^n.$
Relation \eqref{1rfdstft} takes the form
\[ (DS^\ast_{\varphi,\bold u^k}\circ DS_{g,\bold u^k})f=(g,\varphi)f.\]


\subsection{Coordinate transformation}
Let $A_{k,n}=[u_{i,j}]$ be a $k\times n$ matrix with rows $u^i, i=1,...,k$ and $I_{n-k,n-k}$ be the identity matrix. Let $B$ be an $n\times n$  matrix
determined by $A$ and $I_{n-k,n-k}$ so that $Bt=s$, where $$s_1=u_{1,1}t_1+\cdots+u_{1,n}t_n,\; ... ,\; s_k=u_{k,1}t_1+\cdots+u_{k,n}t_n,$$
$s_{k+1}=t_{k+1}, ..., s_n=t_n$. Clearly, this matrix is invertible (regular). Put $C=B^{-1}$ and $\bold e^k=(e_1,...,e_k),$ where $e_1=(1,0,...,0),...,  e_k=(0,...,1)$ are
unit vectors of the coordinate system of $\mathbb R^k$.  Then, with the change of variables
$t=Cs$, and $\eta=C^T\xi$ ($C^T$ is the transposed matrix for $C$), one obtains, for $f\in L^2(\mathbb R^n),$
that (\ref{dd2}) is transformed into:
\begin{equation}\label{dd22}
DS_{g,\bold u^k}f(\tilde y,\xi)=(DS_{g,\bold e^k}h(s))(\tilde y,\eta)=\int_{\mathbb R^n}h(s)\overline{g(\tilde s-\tilde y)}e^{-2\pi i s\cdot\eta}ds,
\end{equation}
where $h(s)=|C|f(Cs)$, $|C|$ is the determinant of $C$, and \eqref{ds} is transformed, for $f\in L^2(\mathbb R^{2n})$, into:
\begin{equation}\label{2ds}
DS^\ast_{g,\bold e^k} F(s)=\int_{\R^n}\int_{\R^k}F(\tilde y,\eta){g(\tilde s-\tilde y)} e^{2\pi is\cdot\eta}d\tilde yd\eta, \quad s\in\R^n.
\end{equation}
\begin{remark}
1. Let $f\in\S^{\alpha}_{\beta}(\mathbb R^n).$ Then
$h(s)=|C|f(Cs)\in\S^{\alpha}_{\beta}(\mathbb R^n).$

2. If $g(s_1,...,s_k)=g_1(s_1)\cdots g_k(s_k)\in (\S^{\alpha}_{\beta}(\mathbb R))^k, (\S^{\alpha}_{\beta}(\mathbb R))^k=\S^{\alpha}_{\beta}(\mathbb R)\times\ldots\times \S^{\alpha}_{\beta}(\mathbb R)$, then
\begin{equation*}\label{1dd2}
DS_{g,u^k} f(\tilde y,\xi): = \int_{\mathbb R^n} f(t)\overline{g_1(u_1\cdot t-y_1)}\cdots \overline{g_k(u_k\cdot t-y_k)} e^{-2\pi it\cdot \xi}dt=
\end{equation*}
$$\int_{\mathbb R^n}  h(s)\overline{g_1(s_1-y_1)}\cdots \overline{g_k(s_k-y_k)} e^{-2\pi is\cdot \mu}ds,
$$
and we call it, the partial short-time Fourier transform.
\end{remark}
\subsection{Continuity properties}

\begin{theorem}\label{ridir}
 By
$$(DS_{g,\bold e^k}h(s))(\tilde y,\eta)=H(\tilde y,\eta)=
\int_{\mathbb R^n}h(s)\overline{g(\tilde s-\tilde y)}e^{-2\pi i s\cdot\eta}ds$$
  is defined a continuous bilinear mapping
$$\S^{\alpha}_{\beta}(\mathbb R^n)\times \S^{\alpha}_{\beta}(\mathbb R^k)\rightarrow
{\S}^{\alpha}_{\beta}(\R^{k+n}),$$
$$(h,g)\mapsto H=DS_{g,\bold e^k}h. $$
\end{theorem}

\begin{proof}
  Let $v,p\in\mathbb N_0^n, \tilde w,\tilde \gamma\in\mathbb N_0^k$,
$\eta\in\mathbb R^n, \tilde y\in\mathbb R^k$. Using \eqref{dd22}, we have
\begin{align*}
J&=\eta^v\tilde y^{\tilde w}\partial_\eta^p\partial_{\tilde y}^{\tilde \gamma} DS_{g,\bold e^k}h(\tilde y,\eta)\\
&=\eta^v\tilde y^{\tilde w}(-2\pi i)^{|p|}\int_{\R^n}(-1)^{|\tilde\gamma|}s^p h(s)\overline{g^{(\tilde \gamma)}(\tilde s-\tilde y)}e^{-2\pi is\cdot\eta}ds\\
&=\eta^v(-2\pi i)^{|p|}\int_{\R^n}(-1)^{|\tilde \gamma|}s^p h(s)\tilde y^{\tilde w}\overline{g^{(\tilde \gamma)}(\tilde s-\tilde y)}e^{-2\pi is\cdot\eta}ds
\\
&=(-2\pi i)^{|p|-|v|}(-1)^{|\tilde \gamma|}
\int_{\R^n}\frac{\partial^{v}}{\partial s^v}\left(s^p h(s)\tilde y^{\tilde w}\overline{g^{(\tilde \gamma)}(\tilde s-\tilde y)}\right)e^{-2\pi is\cdot\eta}ds\\
&=(-2\pi i)^{|p|-|v|}(-1)^{|\tilde \gamma|}\sum_{\tilde j\leq \tilde v}{{v}\choose{\tilde j}}
\int_{\R^n}
\frac{\partial^{v-\tilde j}}{\partial s^{v-\tilde j}}
(s^p h(s))(\tilde y^{\tilde w}\overline{g^{(\tilde\gamma+\tilde j)}(\tilde s-\tilde y)})e^{-2\pi is\cdot\eta}ds.
\end{align*}
Above,  $v-\tilde j=(v_1-j_1,...,v_k-j_k,\tilde{\tilde v})$. In the sequel  we will use
\begin{equation*}
(v-\tilde j)!^{\alpha}\tilde j!^\alpha\leq v!^{\alpha},\;
(\tilde j+\tilde\gamma)!^\alpha\leq 2^{\tilde j+\tilde \gamma}\tilde j!^\alpha\tilde \gamma!^\alpha.
\end{equation*}
We divide integral $\int_{\mathbb R^n}=\int_{|s|\leq1}+\int_{|s|>1}$, denote  corresponding expressions as $J=J_1+J_2$ and continue to estimate the second one $J_2$ in which we put in the numerator and denominator $s^{(2,...,2)}$.
With this preparation we will have the existence of the second integral. This will simplify the estimates inside it.

We put
$J_2=J_{2,1}+J_{2,2}$ and $r_{p}^{v,\tilde \gamma}=(-2\pi i)^{|p|-|v|}(-1)^{|\tilde \gamma|}$. Then,
$$J_{2,1}
=r_{p}^{v,\tilde \gamma}\sum_{\tilde j\leq\tilde v}{{v}\choose{\tilde j}}\int_{|s|>1}
(s^{(2,...,2)}\frac{\partial^{v-\tilde j}}{\partial s^{v-\tilde j}}
(s^p h(s))\overline{g^{(\gamma+\tilde j)}(\tilde s-\tilde y)})
(\tilde y^{\tilde w}-\tilde s^{\tilde w})e^{-2\pi is\cdot\eta}
\frac{ds}{s^{(2,...,2)}}.
$$
$$J_{2,2}=r_{p}^{v,\tilde \gamma}\sum_{\tilde j\leq\tilde v}{{v}\choose{\tilde j}}\int_{|s|>1}
 s^{\tilde w}s^{(2,...,2)}
 \frac{\partial^{v-\tilde j}}{\partial s^{v-\tilde j}}
(s^p h(s))\overline{g^{(\gamma+\tilde j)}(\tilde s-\tilde y)}
e^{-2\pi is\cdot\eta}\frac{ds}{s^{(2,...,2)}}.
$$

The use of positive constants $c_1, c_2, C_1,  C_2,$ below, will be clear from the context.

With suitable $c_1, C_1$ we have
$$\frac{c_1^{|v|+|p|+|\tilde w|+{|\tilde \gamma|}}}{v!^\al p!^{\be}\tilde \gamma!^\alpha \tilde w!^\beta}
|J_{2,1}|/|r_{p}^{v,\tilde \gamma}|\leq
$$
$$
\leq\frac{c_1^{|v|+|p|+|\tilde w|+{|\tilde \gamma|}}}{v!^\al p!^{\be}\tilde \gamma!^\alpha \tilde w!^\beta}\left|\sum_{\tilde j\leq\tilde v}{{v}\choose{\tilde j}}\int_{|s|>1}
(s^{(2,...,2)}\frac{\partial^{v-\tilde j}}{\partial s^{v-\tilde j}}
(s^p h(s))\overline{g^{(\gamma+\tilde j)}(\tilde s-\tilde y)})
(\tilde y^{\tilde w}-\tilde s^{\tilde w})e^{-2\pi is\cdot\eta}
\frac{ds}{s^{(2,...,2)}}\right|
$$
$$
\leq\sum_{\tilde j\leq\tilde v}{{v}\choose{\tilde j}}\int_{|s|>1}\frac{c_1^{|p|+|v-\tilde j|}}{(v-\tilde{j})!^\alpha p!^\beta}|s^{(2,...,2)}
\frac{\partial^{v-\tilde j}}{\partial s^{v-\tilde j}}
(s^p h(s))|
\frac{c_1^{|\tilde  w|+|\tilde\gamma+\tilde j|}|\tilde y-\tilde s|^{\tilde w}|\overline{g^{(\gamma+\tilde j)}(\tilde s-\tilde y)}|}{\tilde w!^\beta(\tilde \gamma+\tilde j)!^\alpha}\frac{ds}{s^{(2,...,2)}}\leq C_{1}.$$

Next, with suitable $c_2, C_2$,
$$\frac{c_2^{|v|+|p|+|\tilde w|+|{\tilde \gamma|}}}{v!^\al p!^{\be}\tilde \gamma!^\alpha \tilde w!^\beta}
|J_{2,2}|/|r_{p}^{v,\tilde \gamma}|
$$
$$
\leq\frac{c_2^{|v|+|p|+|\tilde w|+|{\tilde \gamma|}}}{v!^\al p!^{\be}\tilde \gamma!^\alpha \tilde w!^\beta}\left|\sum_{\tilde j\leq\tilde v}{{v}\choose{\tilde j}}\int_{|s|>1}
 s^{\tilde w}s^{(2,...,2)}
 \frac{\tilde\partial^{v-\tilde j}}{\partial s^{v-\tilde j}}
(s^p h(s))\overline{g^{(\gamma+\tilde j)}(\tilde s-\tilde y)}
e^{-2\pi is\cdot\eta}\frac{ds}{s^{(2,...,2)}}\right|
$$
$$
\leq\sum_{\tilde j\leq\tilde v}{{v}\choose{\tilde j}}\int_{|s|>1}
\frac{c_2^{|p|+|v-\tilde j|+|\tilde w|}}{(v-\tilde{j})!^\alpha p!^\beta\tilde w!^\beta}|s^{(2,...,2)}\tilde s^{\tilde w}
\frac{\partial^{v-\tilde j}}{\partial s^{v-\tilde j}}
(s^p h(s))|
\frac{c_2^{|\tilde \gamma+\tilde j|}\overline{g^{(\gamma+\tilde j)}(\tilde s-\tilde y)}}{(\tilde \gamma+\tilde j)!^\alpha}
\frac{ds}{s^{(2,...,2)}}\leq C_2.
$$
We use \eqref{ngs} and Remark \ref{re1}, that is
\begin{equation*}
\sup_{s\in\mathbb{R}^n, l,q\in\mathbb N_0^n, \tilde w\in\N_0^k}\frac{a^{|q|+|l|+|\tilde w|}|s^{(2,...,2)}\tilde s^{\tilde w}(s^l h(s))^{(q)}|}{q!^\alpha l!^\beta\tilde w!^\beta}<\infty, \mbox{ for some } a>0.
\end{equation*}

\noindent So, with $\sum_{\tilde j\leq \tilde v}{{v}\choose{\tilde j}}\leq 2^v,$ and new constants $c$ and $C$ (which includes the value $\int_{|s|>1}{ds}/s^{(2,...,2)}$), we have
$$|J_{2}|\leq C \sigma_{c}^{\alpha,\beta}(h)\sigma_{c}^{\alpha,\beta}(g).
$$
The estimate on $J_1$ can be done in an analogous fashion, and this completes the proof of the theorem.
 \end{proof}

\begin{proposition}\label{propg}
 Let $f\in \S_\be^\al(\R^n)$ and $g, \varphi\in  \mathcal{S}^\al_\be(\R^k)$. Then the  reconstruction formula \eqref{1rfdstft} holds pointwisely.
\end{proposition}
\begin{proof} Indeed, by using the Parseval identity  for given $f_1,f_2\in L^2(\mathbb R^n)$ and $g,\varphi\in \S_\be^\al(\R^k)$, and after the change of variables as in the representation (\ref{dd22}), that is $h_i(\cdot)=|C|f_i(C\cdot), i=1,2$, we have
\begin{eqnarray}\label{ddd12}&&
 (DS_{g,\bold{u}^k}f_1(\tilde y,\xi),DS_{\varphi,\bold {u}^k}f_2(\tilde y,\xi))_{L^2(\mathbb R^k \times \mathbb R^n)}
\nonumber\\
&=&(DS_{g,\bold{e}^k} h_1(\tilde y,\eta),DS_{\varphi,\bold{e}^k}h_2(\tilde y,\eta))_{L^2(\mathbb R^k \times \mathbb R^n)}\nonumber\\
  &=&( h_1,h_2)_{L^2(\mathbb R^n)}( \overline{g},\overline{\varphi})_{L^2(\mathbb R^k)}.
\end{eqnarray}

We obtain the reconstruction formula
 (\ref{1rfdstft}) as a consequence of (\ref{ddd12}) as in  \cite[Theorem 3.2.1. and Corollary 3.2.3]{Gr}.
\end{proof}
 Now we will consider the continuity properties of \eqref{2ds}.
We fix $g\in\mathcal S^\alpha_\beta(\R^k)$.
\begin{theorem}\label{synthcont}
 By
$DS^\ast_{g,\bold e^k}(H(\tilde y,\eta))(s)=h(s), s\in\R^n$ given by (\ref{2ds}),
  is defined a continuous linear mapping
$$\S^{\alpha}_{\beta}(\R^k\times\mathbb R^n)\rightarrow
{\S}^{\alpha}_{\beta}(\R^{n}),$$
$$H\mapsto h=DS^{\ast}_{g,\bold e^k}H. $$
\end{theorem}

\begin{proof}
We will estimate $
\dfrac{c^{|p+v|}}{p!^\al v!^\be}s^v\partial^p_sh(s)$.
There holds
\begin{align*}
|s^v\partial^p_sh(s)|
&=|s^v\sum_{\tilde j\leq \tilde p}{{ p}\choose{\tilde j}}(2\pi i)^{p-\tilde j}\int_{\R^n}\int_{\R^k}F(\tilde y,\eta)
\eta^{p-\tilde j}{g^{(\tilde j)}(\tilde s-\tilde y)} e^{2\pi i s\cdot\eta}d\tilde{y}d\eta|\\
&=|\sum_{\tilde j\leq \tilde p}{{ p}\choose{\tilde j}}(2\pi i)^{p-\tilde j-v}\int_{\R^n}\int_{\R^k}
\partial_\eta^v(F(\tilde y,\eta)
\eta^{p-\tilde j}){g^{(\tilde j)}(\tilde s-\tilde y)}e^{2\pi {i}s\cdot\eta}d\tilde{y}d\eta|.
\end{align*}
Now it is easy to finish the proof.
\end{proof}
\begin{remark}\label{vind}
This proof shows that one can assume less restrictive conditions on $g$ since we just differentiate $g$. For example, we can only assume that $g\in \mathcal S^\alpha_0(\R^k)$:
\end{remark}
\begin{corollary}
$(DS^\ast_{g,\bold e^k}(H(\tilde y,\eta))(s)=h(s), s\in\R^n$
 defines a continuous bilinear mapping
$$\S^{\alpha}_{\beta}(\R^k\times\mathbb R^n)\times\S^\alpha_0(\R^k)\rightarrow
{\S}^{\alpha}_{\beta}(\R^{n}),$$
$$(H,g)\mapsto h=DS^{\ast}_{g,\bold e^k}H. $$
\end{corollary}
\subsection{$k$-DSTFT and $k$-DSO  on $\S^{\prime\al}_\be$}\label{temperedultra}

Let $\bold u^k=(u_1,\ldots,u_k),$ where $u_i, i=1,\ldots,k$ are independent vectors of $\mathbb S^{n-1}$, and $g\in\S^{\al}_\be(\R^k)$. The continuity results allow us to define $k$-DSTFT of $f\in \S^{\prime\al}_\be(\R^n)$  as an element $DS_{g,\bold u^k}f\in \S^{\prime\al}_{\be}(\R^{k}\times\R^n)$  whose action on test functions is given as a transposed mapping
$$\langle DS_{g,\bold u^k} f,\Phi\rangle=\langle f, DS^\ast_{\overline g, \bold u^k}\Phi\rangle, \qquad \Phi\in \S^{\al}_{\be}(\R^{k}\times\R^n).$$
We use notation $\R^{k}\times\R^n=\R^{k+n}$ just to justify the domain of the above mapping.

Since $g\in\S^{\al}_\be(\R^k)$, one can define
the $k$-DSTFT of an $f\in\mathcal S^{\prime\alpha}_\be(\R^d)$ as
$$DS_{g,\bold u^k}f(\tilde y,\xi)=\langle f(t),g((t\cdot u_1,...,t\cdot u_k)-\tilde y)e^{-2\pi t\cdot\xi}\rangle, \, \tilde y\in \R^k, \, \xi\in\R^n.
$$
This is a direct method of the definition of an integral transform.
We have
\begin{proposition}\label{nova1}
The two definitions of the $k$-DSTFT of an $f\in\mathcal S^{\prime\alpha}_\be(\R^d)$ coincide.
\end{proposition}
\begin{proof}
One has to use the representation formula (\ref{rep11}), then the continuity of
the $P_1(-D)$, of $P_2(x)$ over $\mathcal S^\al_\be(\R^n)$ and then the Fubinni theorem.
\end{proof}

Next, the $k$-DSO $DS^*_{g,\bold u^k}:\S^{\prime\al}_{\be}(\R^{k}\times\R^n)\rightarrow \S^{\prime\al}_{\be}(\R^n)$ can be defined as
\[ \langle DS^*_{g,\bold u^k} F, \varphi\rangle=\langle F, DS_{\overline g, \bold u^k}\varphi\rangle, \quad F\in\S^{\prime\al}_{\be}(\R^{k}\times\R^n), \varphi \in \S_\be^\al(\R^n).\]
We repeat the arguments given above.
Let  $F\in\mathcal S^{\al}_\be(\mathbb R^k\times\R^n)$ be of the form
$$F(\tilde y,\xi)=P_1(D_{\tilde y,\xi})(P_2(\tilde y,\xi)F_0(\tilde y,\xi)) \mbox{ (c.f.(\ref{rep11})), }
$$
where $P_1(D_{\tilde y,\xi})$ and $P_2(\tilde y,\xi)$ are ultradiferential operator over $\R^{k+n}$ of Roumieu class  $\alpha$ and ultradifferential polynomial of Roumieu class $\beta$.
Again, we define $ DS^*_{g,\bold u^k}F$ by a direct method
$$  DS^*_{g,\bold u^k}F(t)=\langle F,g((u_1\cdot t,...,u_k\cdot t)-\tilde y)e^{2\pi i\xi\cdot t}\rangle, \quad t\in\R^n.
$$
We have
\begin{proposition}\label{nova2}
The two definitions of the $k$-DSO
$DS^*_{g,\bold u^k}f$
 of an $f\in\mathcal S^{\prime\alpha}_\be(\R^k\times\R^d)$ coincide.
\end{proposition}

  We immediately obtain:
\begin{proposition} \label{pr1} Let  $g\in\S^{\al}_0(\R^k)$. The $k$-directional short-time Fourier transform, $DS_{g,\bold u^k}:\S^{\prime\al}_\be(\R^n)\to \S^{\prime\al}_\be(\R^k\times\R^n)$ and the  synthesis operator $DS^{*}_{g,\bold u^k}:$  $ \S^{\prime\al}_\be(\R^k\times\R^n)\to\S^{\prime\al}_\be(\R^n)$ are  continuous linear maps. \end{proposition}

The following theorem connects the $k$-DSTFTs  with respect to different windows.
\begin{theorem}\label{d444} Let $\bold u^k=(u_1,\ldots,u_k),$ where $u_i, i=1,\ldots,k$ are independent vectors of $\mathbb S^{n-1}$.
Let  $\varphi, g, \gamma$ belong to $\mathcal S^\al_0(\mathbb R^k)$ where $\gamma$ is the synthesis
window for $g$. Let $f\in\mathcal S^{\prime\al}_\be(\mathbb R^n)$, then
$$DS_{\varphi,\bold u^k}f(\tilde x,\eta)=(DS_{g,\bold u^k}f(\tilde s,\zeta))*(DS_{\varphi,\bold u^k}\gamma(\tilde s,\zeta))(\tilde x,\eta), \quad \tilde x,\tilde s\in\R^k, \ \eta, \zeta\in\R^n.
$$
\end{theorem}
\proof We follow the proof in \cite{APS}.  By  \eqref{dd22}, it is enough to prove the assertion for $\bold e^k$. Let $F\in\S^{\prime\al}_{\be}(\mathbb R^{k}\times\R^n)$. By the continuity arguments, we can assume that $f=F\in L^2(\R^k\times\R^n).$
Then
\begin{eqnarray*} DS_{\varphi,\bold e^k}(DS_{\gamma,\bold e^k}^{*}F)(\tilde x,\eta) &=& \int_{\mathbb R^n}\big(\int_{\mathbb R
^{n}}\int_{\mathbb R^k}F(\tilde y, \xi)
\gamma(\tilde t-\tilde y) e^{2\pi i\xi \cdot t}d\tilde y d\xi\big)
\overline {\varphi(\tilde t-\tilde x)} e^{-2\pi it\cdot\eta}dt
\\&=& \int_{\mathbb R^n}\int_{\mathbb R
^{k}}(\int_{\mathbb R^n}\gamma(\tilde t)
\overline{\varphi(\tilde t-(\tilde x-\tilde y))}
 e^{-2\pi it\cdot(\eta-\xi)}dt)F(\tilde y,\xi)d\tilde y d\xi\\&=&\int_{\mathbb R^n}\int_{\mathbb R^{k}}F(\tilde y,\xi)DS_{\varphi,\bold e^k}\gamma(\tilde x-\tilde y,\eta-\xi)d\tilde yd\xi.
\end{eqnarray*}
Now, we put $F=DS_{g,\bold e^k}f$ and obtain
\begin{equation}\label{dd333}
DS_{\varphi,\bold e^k}f(\tilde x,\eta)=(DS_{g,\bold e^k}f(\tilde s,\zeta))*(DS_{\varphi,\bold e^k}\gamma(\tilde s,\zeta))(\tilde x,\eta).
\end{equation}
\qed


\section{Directional wave fronts}\label{se3}
In order to detect singularities determined by the hyperplanes orthogonal to vectors $ u_1,..., u_k$ using the $k$-DSTFT, we introduce $k$-directional regular sets and wave front sets for GS ultradistributions. Theorem \ref{d444} guaranties that the wave front set will not depend on the used window. Again, we simplify our exposition by the use of  \eqref{dd22}
and transfer the  STFT in $\bold u^k$ direction to STFT in  $\bold e^k$ direction.

Let $k=1$ and $y_0=y_{0,1}\in\mathbb R$. Put
$
\Pi_{e_1,y_0,\varepsilon}=\{t\in \mathbb R^n: |t_1- y_0|<\varepsilon\}.
$
It is an area of $\mathbb R^n$ between two hyperplanes orthogonal to $e_1$,
$$
 \Pi_{e_1,y_0,\varepsilon}= \bigcup_{y\in (y_0-\varepsilon,y_0+\varepsilon)} P_{e_1,y},
\;\; \; (y_0=(y_0, 0,\ldots,0), y=(y,0,...,0)),$$
 and
$P_{e_1,y}$ denotes the hyperplane orthogonal to $e_1$ passing through  $y$.

We keep the notation of Section \ref{se2}.
Put
$$\Pi_{\bold e^k,\tilde y,\varepsilon}=\Pi_{e_1,y_1,\varepsilon}\cap\ldots\cap\Pi_{e_k,y_k,\varepsilon}, \quad \Pi_{\bold e^k,\tilde y}=\Pi_{e_1,y_1}\cap\ldots
\cap\Pi_{e_k,y_k}.
$$
The first set is a paralelopiped  in $\mathbb R^k$ so that in $\mathbb R^n$ it is determined by $2k$ finite edges while the other edges are infinite. The set
$\Pi_{\bold e^k,\tilde y}$ equals  $\mathbb R^{n-k}$ translated by vectors
$\vec y_1,\ldots,\vec y_k.$ We will call it $n-k$-dimensional element of
$\mathbb R^n $ and denote it as $P_{\bold e^k,\tilde y}\in\mathbb R^k.$
 If $k=n$, then this is just the point $y=(y_1,\ldots,y_n).$

In the sequel, $\alpha>1.$

\begin{definition}\label{wp}
Let $f\in \mathcal S^{\prime\al}(\mathbb R^n)$. It is said that $f$ is $k$-directionally  microlocally regular  at $(P_{\bold e^k,\tilde y_0},\xi_0)\in
\mathbb R^n\times (\mathbb R^n\setminus \{0\})$, that is, at every point of the form $((\widetilde y_0, \cdot),\xi_0)$ ($\cdot$ denotes an arbitrary point of $\R^{n-k}$)
if there exists $g\in \mathcal D^{\al}(\mathbb R^k)$, $g(\tilde 0)\neq 0$, the product of
open balls $L_r(\tilde y_0)=L_r(y_{0,1})\times...\times L_r(y_{0,k})\in\mathbb R^k$, a  cone $\Gamma_{\xi_0}$ and  there exist
$N\in\mathbb N$  and $C_{N}>0$
such that
\begin{equation}\label{rhh}
\sup_{\tilde y \in L_r(\tilde y_0),\,\xi \in\Gamma_{\xi_0}}|DS_{g, \bold e^k}f((\tilde y,\cdot),\xi)|
=\sup_{\tilde y\in L_r(\tilde y_0),\,\xi \in\Gamma_{\xi_0}}|\mathcal F
(f(t)\overline{g(\tilde t-\tilde y)})(\xi)|\leq C_{N} e^{-N|\xi|^{1/\al}}.
\end{equation}
\end{definition}
Note that for $k=n$ our definition is the  classical H\" ormander's  definition of regularity, \cite[Section 8]{hor}.
\begin{remark}\label{wfud}
a) If $f$ is $k$--directionally  microlocally regular at  $(P_{\bold e^k,\tilde y_0},\xi_0)$, then there exists an open ball  $ L_r(\tilde y_0)$ and an open cone $\Gamma\subset\Gamma _{\xi_0}$
so that $f$ is $k$--directionally  microlocally regular at  $(P_{\bold e^k,\tilde z_0},\theta_0)$ for any $\tilde z_0\in L_{r}(\tilde y_0)$ and $\theta_0 \in \Gamma.$ This implies that the union  of all $k$--directionally  microlocally regular points $(P_{\bold e^k,\tilde z_0},\theta_0)$, $((\tilde z_0,\cdot),\theta_0)\in (L_{r}(\tilde y_0)\times\RR^k)\times\Gamma$ is an open set of $\mathbb R^n\times(\mathbb R^n\setminus\{0\})$.

b) Denote by $Pr_{k}$ the projection of $\mathbb R^n$ onto $\mathbb R^k$. Then, the  $k$--directionally  microlocally regular
point $(P_{\bold e^k,\tilde y_0},\xi_0)$, considered in $\mathbb R^n\times(\mathbb R^n\setminus\{0\})$ with respect to the first $k$ variables, equals $(Pr_k^{-1}\times I_\xi)(P_{\bold e^k,\tilde y_0},\xi_0)$ ($I_\xi $ is the identity matrix on $\mathbb R^n$).

As standard, $k$-directional wave front  is defined as the complement in
$\mathbb R^k\times(\mathbb R^n\setminus\{0\})$ of all    $k$--directionally  microlocally regular points $(P_{\bold e^k,\tilde y_0},\xi_0)$, and we denoted as
$WF_{\bold e^k}(f).$
\end{remark}

\begin{proposition}
The set
$WF_{\bold e^k}(f)$
is closed  in $\mathbb R^k\times(\mathbb R^n\setminus\{0\})$ (and $\mathbb R^n\times (\mathbb R^n\setminus \{0\})$).
\end{proposition}

$B_r(\tilde 0)$  denotes a closed ball in $\mathbb R^k$ with center at  zero  and radius $r>0.$
The following theorem relates  sets of $k$--directionally  microlocally regular points for two $k$-DSTFT of GS ultradistributions.

\begin{theorem} \label{nwr}  If (\ref{rhh}) holds for some $g\in\mathcal D^{\al}(\mathbb R^k)$, then it holds for every $h\in\mathcal D^{\al}(\mathbb R^k),$ $(h(\tilde 0)\neq 0)$ supported by a ball
$B_\rho(\tilde 0)$, where $\rho\leq\rho_0$ and $\rho_0$ depends on $r$ in  (\ref{rhh}).
\end{theorem}
\begin{proof}

We follow the  idea of our  proof in \cite{APS}. Compact supports of  $g$ and $h$ simplify the integration. Moreover, by the structural theorem,  we know that $f=P_0(D)F$, where $F$ is a continuous function of sub-exponential growth:
\begin{equation}\label{gr1}
 (\forall  a>0, \quad \exists C_a>0) \quad  F(x)\leq C_ae^{a|\xi|^{1/\alpha}}
\end{equation}
 and $P_0(D)$ satisfies (\ref{RP}) and  \eqref{ulpobound}. So, we can use the technique of oscillatory integral, transfer the differentiation from $f$ on other factors in integral expressions and, from the beginning, assume that   $f$ is a continuous function which satisfies \eqref{ulpobound}.

We use Theorem \ref{d444}, that is, the form (\ref{dd333}).   Assume that (\ref{rhh}) holds. We repeat from \cite{APS} the constructions of balls. Without that one can not follow our new
 estimates.

So,  $\gamma$ is chosen so that
 $\mbox{ supp }\gamma\subset B_{\rho_1}(\tilde 0)$ and   $\rho_1<r-r_0$.
Let  $h\in\mathcal D^{\al}(\mathbb R^k)$ and $\mbox{ supp }h\subset B_{\rho}(\tilde 0)$.
  Our aim is to  find $\rho_0$ such that  (\ref{rhh}) holds for $DS_{h,\bold e^k}f(\tilde x,\eta)$, with
$\tilde x\in B_{r_0}(\tilde y_0), \eta\in\Gamma_1\subset\subset \Gamma_{\xi_0},$
for $ \rho\leq \rho_0$
($\Gamma_1\subset\subset \Gamma_{\xi_0}$ implies that $\Gamma_1\cap \mathbb S^{n-1}$
is a compact subset of  $\Gamma_{\xi_0}\cap \mathbb S^{n-1}$).

Next,
$$
|\tilde p|\leq \rho_1,\;\; |\tilde x-\tilde y_0|\leq r_0 \mbox{ and }\;\; |\tilde p-((\tilde x-\tilde y_0)-(\tilde y-\tilde y_0))|\leq \rho
$$
\begin{equation}\label{suporti}
\Rightarrow |\tilde y-\tilde y_0|\leq \rho+\rho_1+r_0.
\end{equation}
So, we choose $\rho_0$ such that
$\rho_0+\rho_1<r-r_0$
and
\begin{equation}\label{suporti2}\rho+\rho_1+r_0<r \;\mbox{ holds for }\;
 \rho\leq\rho_0.
\end{equation}

Let $\Gamma_1\subset\subset \Gamma_{\xi_0}$.
Then, there exists $c\in (0,1)$ such that
\begin{equation}\label{gam}
\eta\in \Gamma_1, |\eta|>1 \mbox{ and } |\eta-\xi|\leq c|\eta|\Rightarrow \xi \in\Gamma_{\xi_0}; \;\;
|\eta-\xi|\leq c|\eta|\Rightarrow |\eta|\leq (1-c)^{-1}|\xi|.
\end{equation}
Let $\tilde x\in B_{r_0}(\tilde y_0), \eta \in\Gamma_1$.
Then
$$
|DS_{h,\bold e^k}f((\tilde x,\cdot),\eta)|
=\left|\int_{\mathbb R^k}\int_{\mathbb R^n}DS_{g,\bold e^k}f((\tilde y,\cdot),\xi)DS_{h,\bold e^k}\gamma(\tilde x-\tilde y,\eta-\xi)d\xi d\tilde y\right|.
$$
Consider $$ J_1=\int_{\mathbb R^n}DS_{g,\bold e^k}f((\tilde y,\cdot),\eta-\xi)d\xi \mbox{ and }\;  J_2=\int_{\mathbb R^n}DS_{h,\bold e^k}\gamma(\tilde x-\tilde y,\xi)d\xi.$$
We choose an elliptic ultradifferential operator $P(D)$ so that
$P(\xi)\geq e^{(a+1)|\xi|^{1/\alpha}},$ where $a$ is from (\ref{gr1}). Then
$$J_1=\int_{\mathbb R^n}\int_{\mathbb R^n}
\frac{f(t)}{P(2\pi t)}\overline{g(\tilde t-\tilde y)}P(D_\xi)(e^{-2\pi i t\cdot(\eta-\xi)})dtd\xi.
$$
 This integral  diverges with respect to $\xi$, while $J_2$ converges because
\begin{eqnarray*}J_2&=&\int_{\mathbb R^n}\int_{B_{\rho_1}(\tilde 0)}
\frac{\gamma(\tilde p)\overline{h(\tilde p-(\tilde x-\tilde y))}}{P(-2\pi\xi)}P(D_p)(e^{-2\pi i p\cdot\xi})dpd\xi
\\&=&
\int_{\mathbb R^n}\int_{B_{\rho_1}(\tilde 0)}
P(D_p)(\frac{\gamma(\tilde p)\overline{h(\tilde p-(\tilde x-\tilde y))}}{P(-2\pi\xi)})e^{-2\pi i p\cdot\xi}dpd\xi.
\end{eqnarray*}
Rewrite
$$|DS_{h,\bold e^k}f((\tilde x,\cdot),\eta)|=\int_{\mathbb R^k}|(\int_{|\eta-\xi|\leq c|\eta|}+\int_{|\eta-\xi|\geq c|\eta|})(...)d\xi|d\tilde y=I_1+I_2.
$$
Then,
$$I_1\leq \int_{\mathbb R^k}\left(\sup_{|\eta-\xi|\leq c|\eta|}
|DS_{g,\bold e^k}f((\tilde y,\cdot),\eta-\xi)|\int_{|\eta-\xi|\leq c|\eta|}
|DS_{h,\bold e^k}\gamma(\tilde x-\tilde y,\xi)|d\xi\right)d\tilde y.
$$
Using (\ref{suporti}), (\ref{suporti2}) and \eqref{gam} we obtain
\begin{equation}\label{dod1}
\sup_{\tilde x\in B_{r_0}(\tilde y_0),\,\eta\in\Gamma_1}e^{N|\eta|^{1/\al}}I_1\leq \int_{B_r(\tilde y_0)}\left (\sup_{\xi\in \Gamma_{\xi_0}}
|DS_{g,\bold e^k}f((\tilde y,\cdot),\xi)|e^{N(1-c)^{-1}|\xi|^{1/\al}}\right .
\end{equation}
$$\left .\times \int_{|\xi|\geq (1-c)|\eta|}
|DS_{h,\bold e^k}\gamma(\tilde x-\tilde y,\xi)|d\xi\right )d\tilde y.
$$
Now by the finiteness of $J_2$, we obtain that $I_1$ satisfies the necessary estimate of (\ref{rhh}).

Now we consider $I_2$ with the new explanations.
$$I_2\leq \int_{\mathbb R^k}\left|\int_{|\xi|\geq c|\eta|}
DS_{g,\bold e^k}f((\tilde y,\cdot),\eta-\xi)DS_{h,\bold e^k}\gamma(\tilde x-\tilde y,\xi)d\xi\right| d\tilde y.
$$

Let $K=\{\xi: |\eta-\xi|\geq c|\eta|\}$.
Denote by $\kappa^0_d, 0<d<1,$ the characteristic function of
 $K_{d}=\bigcup_{\xi\in K}L_d(\xi)$, that is, $K_d$ is open $d$-neighborhood of $K.$
Then, put $$\kappa_d={\kappa}^0_{d}*\varphi_{d},$$
where $\varphi_d=\frac{1}{d^n}\varphi(\cdot/d)$,
$\varphi\in\mathcal D^{\al}(\mathbb R^n)$ is non-negative, supported by the ball $B_1(0)$ and equals
$1$ on $B_{1/2}(0).$ This construction implies that $\kappa_d$ equals one on $K$,  is supported by $K_{2d}.$
Moreover,  all the derivatives of $\kappa_d$ are bounded.
We note that
$$|\int_K...d\xi|=|\int_{K_{2d}}\kappa_d(\xi)...d\xi| +|\int_{K_{2d}\cap\{\xi: |\eta-\xi|\leq\eta\}}\kappa_d(\xi)...d\xi|
$$
Then,
$$\sup_{\tilde x\in B_{r/2}(\tilde y_0),\,
 \eta\in \Gamma_1}I_2\leq
 \int_{\mathbb R^k}
 |\int_{\R^n}\kappa_d(\xi)
DS_{g,\bold e^k}f((\tilde y,\cdot),\eta-\xi)
DS_{h,\bold e^k}\gamma(\tilde x-\tilde y,\xi)d\xi|d\tilde y
$$
$$
+ \int_{\mathbb R^k}|
\int_{K_{2d}\cap\{\xi: |\eta-\xi|\leq\eta\}}\kappa_d(\xi)
DS_{g,\bold e^k}f((\tilde y,\cdot),\eta-\xi)
DS_{h,\bold e^k}\gamma(\tilde x-\tilde y,\xi)d\xi| d\tilde y=I_{2,1}+I_{2,2}.
$$
We continue with $I_{2,1}$.
For every $\tilde x\in B_{r_0}(\tilde y_0)$ and $\eta\in\Gamma_1$ and using  \eqref{ulpobound} we see that all  integrals on the right hand side of
$$\sup_{\tilde x\in B_{r/2}(\tilde y_0),\,
 \eta\in \Gamma_1}e^{N|\eta|^{1/\al}}I_{2,1}
 $$
 $$\leq
\int_{\mathbb R^k}|\int_{\mathbb R^n_\xi}
(\int_{\mathbb R^n_t}\frac{|f(t)|}{P(2\pi t)}|\overline{g(\tilde t-\tilde y)}|dt
\big(\frac{e^{N|\xi|^{1/\al}}}{e^{N|\eta-\xi|^{1/\al}}}P(D_\xi)
\frac{\kappa_d(\xi)}{P(-2\pi\xi)}\big)
$$
$$ \int_{\mathbb R^n_p}|P(D_p)\big(\gamma(\tilde p)\overline{h(\tilde p-(\tilde x-\tilde y))}\big)|dp \big)d\xi| d\tilde y
$$
are finite.

Now we treat $I_{2,2}$ in the same way as $I_{1}$ since integration goes over a subset of
$\{\xi: |\eta-\xi|\leq c|\eta|\}$. Only a bounded factor $\kappa_d$ appears in
the last integral in (\ref{dod1}):
$$\int_{|\xi|\geq (1-c)|\eta|}\kappa_d(\xi)
|DS_{h,\bold e^k}\gamma(\tilde x-\tilde y,\xi)|d\xi.
$$
This  gives
$$\sup_{\tilde x\in B_{r/2}(\tilde y_0),\,
 \eta\in \Gamma_1}e^{N|\eta|^{1/\al}}I_{2,2}<\infty.
$$
This completes the proof of the theorem.
\end{proof}

The next corollary is a modification of the one in the distribution theory (see \cite{APS}).
\begin{corollary}
\label{suz}
Let $g\in\mathcal D^{\al}(\mathbb R^k)$ with $\mbox{supp }g \subset B_a(\tilde 0)$, have  synthesis window $\gamma$  with $ \mbox{supp } \gamma\subset B_{\rho_1}(\tilde 0)$ and
$\rho_1\leq a.$ Then
\begin{equation}\label{2rh}
\sup_{\tilde y\in L_{2r}(\tilde y_0),\,\xi \in\Gamma_{\xi_0}}|DS_{g, \bold e^k}f((\tilde y,\cdot),\xi)|\leq C_Ne^{-N|\xi|^{1/\al}}.
\end{equation}
 Moreover, assume that $a<r.$ Then, for any $h\in\mathcal D^{\al}(\mathbb R^k)$ with support $B_{\rho}(\tilde 0)$, $\rho<a$, there exists $r_0$ and $\Gamma_1\subset\subset \Gamma_{\xi_0}$ such that (\ref{2rh}) holds for $DS_{h,\bold e^k}f((\tilde x,\cdot),\eta)$ with the supremum over
$\tilde x\in B_{r_0}(\tilde y_0)$ and $ \eta \in \Gamma_1$.
\end{corollary}
\begin{proof}
Similarly
as in (\ref{suporti}),
$$|\tilde y -\tilde y_0|\leq \rho+\rho_1+r_0<\rho+a-r_0+r_0=a+\rho<2r.
$$
This implies $|\tilde y -\tilde y_0|<2r,$ so that the supremum in the estimate of $I_1$ holds. The proof now can be performed in the same way as in Theorem \ref{nwr}.
\end{proof}

With the standard  proof we have
\begin{corollary}\label{w1}
If
$(P_{\bold e^k,\tilde y},\xi)$ is a  $k$-directionally microlocally regular point of $f\in\mathcal S'^{\al}(\mathbb R^n)$ for every $\xi\in\mathbb R^n\setminus \{0\}$,
then  $ f\in\mathcal E^{\al}(\mathbb R^n).$
\end{corollary}


\subsection{Relations with the partial wave front}
Recall again that  the partial wave front of an $f\in\DD^{\prime}(\RR^n)$  was not considered in the literature. So, it is clear that this was not done for ultradistribution spaces in particular, for GS spaces of ultradistributions.  For the purpose of this investigations, we will consider
the set of $k$-microlocally regular points (distinguishing them from the $k$-directionally microlocally regular points, for a moment) for an $f\in\mathcal S^{\prime\alpha}$:
The point $((\widetilde y_0,\widetilde{\widetilde{y}}_0),\xi_0)\in(\RR^k\times\RR^{n-k})\times (\RR^n\setminus\{0\})$ is $k$-microlocally regular for $f$ if
there exists $ \chi\in\DD^{\alpha}(\RR^k)$ so that $ \chi(\tilde y_0)\neq 0$  and a cone $\Gamma_{\xi_0}$ around $\xi_0$ so that there exist $N\in\NN$ and $C_{N,\chi}>0$  such that
\begin{equation}\label{sftkrn137}
|\mathcal{F}(\chi(\tilde y) f(y))(\xi)|\leq C_{N,\chi}e^{-N|\xi|^{1/\alpha}}, \quad \xi\in\Gamma_{\xi_0}, \, y=(\tilde y, \tilde{\tilde y})\in\R^k\times \R^{n-k}.
\end{equation}
Since, in this definition, $\chi$ does not depend on $\widetilde{\widetilde y},$
we  will write in the sequel that $f$ is $k$--microlocally regular at
$((\widetilde y_0,\cdot),\xi_0)$.
\begin{remark} The implication (\ref{rhh}) $\Rightarrow$ (ii) is clear. We will prove, as a part of the next theorem ($(ii)\Rightarrow {(i)}$), that the opposite implication also holds, which means that the two notions coincide.
\end{remark}
\begin{theorem}
Let $f\in\mathcal S^{\prime\alpha}(\RR^d)$ and $((\widetilde y_0,\cdot),\xi_0)\in(\RR^k\times\RR^{n-k})\times(\RR^n\backslash\{0\})$. The following conditions are equivalent.

(i) $((\widetilde y_0,\cdot),\xi_0)\not\in WF_{\bold e^k}(f)$.

(ii) There exist a compact neighbourhood $\widetilde K$ of $\widetilde y_0$ and a cone neighbourhood $ \Gamma$ of $\xi_0$ such that there exist $N\in\NN$ so that for every
${\chi}\in\mathcal D^{\al}({\widetilde K})$ there exists $C_{N, \chi}>0$ such that (\ref{sftkrn137}) is valid.

(iii) There exist a compact neighbourhood $\widetilde K$ of $\widetilde y_0$ and a cone neighbourhood $\Gamma$ of $\xi_0$ such that there exist $N\in\N$, $h>0$  and  $C_{N,h}>0$   such that
\[
    |DS_{\chi, \bold e^k} f((\widetilde y,\cdot),\xi)|\leq C_{N,h}
    \sup_{p\in\N_0^n}\frac{h^{|p|}}{p!^\al}\|D^{p}\chi\|_{L^{\infty}(\RR^k)}e^{-N|\xi|^{1/\alpha}},
    \]
    \[\forall \widetilde y\in \widetilde K,\,\, \forall \xi\in\Gamma,\,\, \forall \chi\in \DD^{\alpha}(\widetilde K-\{y_0\}),
   \]
    where $\widetilde K-\{\widetilde y_0\}=\{\widetilde y\in\RR^k|\,
    \widetilde y+\widetilde y_0\in \widetilde K\}$.

(iv)There exist a compact neighborhood $\widetilde K$ of $\widetilde y_0$, a cone neighborhood $\Gamma$ of $\xi_0$ and $\chi\in\DD^{\alpha}(\RR^k)$, with $\chi(\widetilde 0)\neq 0$ such that there exist $N\in\NN$ and  $C_{N, \chi}>0$ such that
    \[
    |DS_{\bold e^k,\chi}f((\widetilde y,\cdot),\xi)|\leq C_{N,\chi}e^{-N|\xi|^{1/\alpha}},\,\, \forall \tilde y\in \widetilde K,\, \forall \xi\in\Gamma.
    \]

\end{theorem}

\begin{proof} The proof follows the steps of  those one in \cite{PP}.

 $(i)\Rightarrow (ii)$
The fact that $((\widetilde y_0,\cdot),\xi_0)\not\in WF_{\bold e^k}(f)$ implies the existence
of $\chi\in\DD^{\alpha}(\RR^k)$ ($\chi(\cdot)=g(\cdot-\widetilde y_0)$) with $\chi(\tilde y_0)\neq 0$ and a cone neighborhood $\Gamma_{\xi_0}$ of $\xi_0$ for which (\ref{sftkrn137}) is valid for
$\xi\in\Gamma_{\xi_0}$. There exists a compact neighborhood
$\widetilde K$ of $\tilde y_0$ where $\chi$ never vanishes. Fix a cone neighborhood $\Gamma$ of $\xi_0$ such that $\overline{\Gamma}\subseteq \Gamma_{\xi_0}\cup\{0\}$.
 Following  the proof of Lemma 8.1.1 in \cite{hor} one can show that there exist $N\in\NN$  and $\psi\in\DD^{\alpha}(\widetilde K)$  such that $|\mathcal{F}(\psi\chi f)(\xi)|\leq C_{N,\psi,\chi}e^{-N|\xi|^{1/\alpha}}$,
 $\forall \xi\in\ \Gamma$. Then $(ii)$   follows since
 $\psi f=(\psi/\chi)\chi f$ where $\psi/\chi\in\mathcal D^{\alpha}(\widetilde K)$.

$(ii)\Rightarrow (iii)$ By $(ii)$, (\ref{sftkrn137}) implies that there exist $N\in\NN$, the set
 $H_N=\{e^{-N|\xi|^{1/\alpha}} e^{-i\xi \cdot}f|\, \xi\in \Gamma\}$ is weakly bounded in
 $\mathcal D'^{\alpha}(\widetilde B)$. So it is equicontinuous as
 $D^{\alpha}(\widetilde B)$ is barrelled. Let $\widetilde K=\widetilde B_{\widetilde y_0}(r/2)$. For each
 $\chi\in\mathcal D^{\alpha}(\widetilde K-\{y_0\})$ and $\widetilde y\in \widetilde K$ the function
 $$\widetilde t\mapsto \chi(\widetilde t-\widetilde y)$$ is in $\mathcal D^{\alpha}(\widetilde B)$ and the equicontinuity of $H_N$ implies the existence of $C_N>0$ and $h>0$
   such that
\beqs
|\langle e^{-i\xi t} f(t),\overline{\chi(\widetilde t-\widetilde y)}\rangle|&\leq& C_N
e^{-N|\xi|^{1/\alpha}}\sup_{p\in\N_0^n,\widetilde t,\widetilde y\in\widetilde K,}\
\frac{h^{|p|}}{p!^{\alpha}}|D^{m}\chi(\widetilde t-\widetilde y)|\\
&=&C_N\sup_{p\in\N_0^n}\frac{h^{|p|}}{p!^\alpha}\|D^{p}\chi\|_{L^{\infty}(\RR^d)}
e^{-N|\xi|^{1/\alpha}},\,\, \forall \xi\in\Gamma,\, \forall \widetilde y\in \widetilde K,
\eeqs
which implies the validity of $(iii)$.

 $(iii)\Rightarrow (iv)$ is simple and skipped.

 Using the estimate in $(iv)$ with $\widetilde y=\widetilde y_0$, $(iv)\Rightarrow (i)$ simply follow. This completes the proof.
\end{proof}


{Acknowledgements}: {This paper was
supported by the project "\textit{Time-frequency methods}", No. 174024 financed by the
Ministry of Science, Republic of Serbia,  by the project "\textit{Localization in phase space: theoretical and numerical aspects}", No. 19.032/961-103/19 funded by MNRVOID Republic of Srpska}  and by the bilateral project “\textit{Microlocal analysis and applications}" between the Macedonian and Serbian academies of sciences and arts.

\end{document}